\documentclass[12pt,a4paper]{article}
\usepackage{amsmath,amsfonts,amssymb,amsthm,}
\usepackage{color}
\usepackage[english]{babel}
\usepackage{vmargin}

\setmarginsrb{23mm}{20mm}{23mm}{20mm}{10mm}{10mm}{10mm}{10mm}

\newtheorem{theorem}{Theorem}

\newtheorem{theo}{Theorem}

\newtheorem{lem}{Lemma}

\theoremstyle{definition}

\theoremstyle{remark}

\newcommand{\D}{\mathbb{D}}
\newcommand{\R}{\mathbb{R}}

\newcommand{\T}{\mathbb{T}}
\newcommand{\ud}{\mathrm {d}}
\newcommand{\Hp}{\mathcal{H}^p}
\newcommand{\dist}{\operatorname{dist}}
\newcommand{\diam}{\operatorname{diam}}

\definecolor{blau}{rgb}{0.1,0.0,0.9}
\definecolor{violet}{rgb}{0.54, 0.17, 0.89}
\newcommand{\blue}{\color{blau}}

\newcommand{\kom}[1]{}
\renewcommand{\kom}[1]{{\bf \blue /#1/}}

\newcounter{komcounter}
\numberwithin{komcounter}{section}

\title{Hardy spaces for quasiregular mappings and composition operators}

\author{Tomasz Adamowicz{\small$^1$},  Mar\'ia\ J.\ Gonz\'alez{\small$^2$}}

\date{}

\begin{document}
	
	\maketitle
	
	\baselineskip=6mm
	\parskip=2.5mm
	\let\thefootnote\relax\footnote {\textit{2010 Mathematics Subject Classification:} 30H10, 30H20, 30C20}
	\let\thefootnote\relax\footnote{\textit{Keywords and phrases:}  Carleson measure,  composition operators, Hardy spaces, quasiregular mappings, quasiconformal symbols}
	
	\begin{abstract}
		We define Hardy spaces $\mathcal{H}^p$ for quasiregular mappings in the plane, and show that for a particular class of these mappings many of the classical properties that hold in the classical setting of analytic mappings still hold. This particular class of quasiregular mappings can be characterised in terms of composition operators when the symbol is quasiconformal. Relations between Carleson measures and Hardy spaces play an important role in the discussion. This program was initiated and developed for Hardy spaces of quasiconformal mappings by Astala and Koskela in 2011 in their paper \textit{$\mathcal{H}^p$-theory for Quasiconformal Mappings}~\cite{AK}.
	\end{abstract}

	\footnotetext{{\small $^1$}\!\!\! T.A. was supported by a grant of National Science Center, Poland (NCN),  UMO-2017/25/B/ST1/01955.} 
	\footnotetext{{\small $^2$} M.J.G. was supported by Plan Propio de la UCA 2019 and Plan Nacional I+D grant no. MTM2017-85666-P, Spain.  
	
	This work was partially supported by the grant 346300 for IMPAN from the Simons Foundation and the matching 2015-2019 Polish MNiSW fund.}

	\section*{Introduction}
	
	Let $\mathbb{D}$ denote the unit disc $\{z\in \mathbb{C};~|z|<1\}$ and  $\mathbb {T}=\partial \mathbb{D}$. Furthermore, for  
	 a given point $\xi\in \mathbb{T}$ and $c>1$, let us denote  the cone with vertex at $\xi\in \mathbb{T}$ as follows
	\[
	\Gamma(\xi) = \{z\in \mathbb{D};~|z-\xi|< c~ (1-|z|)\}.
	\]

An analytic function $f$ in $\mathbb{D}$ is in the Hardy space $\mathcal{H}^p$ for $0<p<\infty $, if
$$
\sup_{0<r<1} ~\frac{1}{2\pi} \int_0^{2\pi}|f (re^{i\theta})|^p ~\ud\theta = \|f\|_{\mathcal{H}^p}^p < \infty.
$$

 The $\mathcal{H}^p$  theory for analytic functions is very well understood (see for instance \cite {G}, \cite{D} or \cite{BM}). It is well known that if  a function  $f\in \mathcal{H}^p, ~0<p<\infty$, then  $f$
 has almost everywhere  non-tangential boundary limit  $f(e^{i\theta})\in L^p (\mathbb {T})$, and 
 $ \frac{1}{2\pi} \int_0^{2\pi}|f (e^{i\theta})|^p  ~\ud\theta= \|f\|_{\mathcal{H}^p}^p $.
 Moreover, the non-tangential maximal function of $f$, defined as $f^*(\xi)= \sup_{z\in \Gamma (\xi)} |f(z)|$ satisfies $f^* \in  L^p (\mathbb {T})$, with 
 $\|f^*\|_{L^p}\leq c ~\|f\|_{\mathcal{H}^p}$.\\
  \indent
  If $p\geq 2$, then by a theorem of Littlewood and Paley \cite{LP} the derivatives of an $\mathcal{H}^p$ function also satisfy that 
  \begin{equation}\label{area}
  \int_{\mathbb{D}}|f'(z)|^p (1-|z|)^{p-1}~\ud m<\infty.
  \end{equation}
  
  The aim of this note is to analyze these results in the quasiregular setting. A mapping $f:\mathbb{D}\rightarrow \mathbb{R}^2$  is called $K$-\textit{quasiregular} for $K\geq 1$, if  $f$ belongs to the Sobolev space $W_{loc}^{1,2}(\mathbb{D}, \R^2)$ and the distortion inequality
  \[
  |Df(z)|^2\leq  K Jf (z)
  \] 
  holds for almost every $z\in \mathbb{D}$.\\
 \indent
 If in addition we require $f$ to be a homeomorphism, then we say that $f$ is $K$-\textit{quasi\-con\-formal}. For comprehensive introductions to the topic and further references we refer, for instance, to~\cite{AIM}, \cite{LV} and~\cite{va}.
 
  The theory of Hardy spaces for quasiconformal mappings  has been developed by Astala and Koskela in their seminal paper \cite{AK}. Though their results hold also in higher dimensions, we will restrict our attention to dimension $2$ only, leaving the discussion of higher dimensional cases to separate further studies.\\
 \indent
  Astala and Koskela show that many of the properties of  analytic $\mathcal{H}^p$ functions are also shared by $\mathcal{H}^p$-quasiconformal maps, in particular the ones described above related to the maximal function and to the existence of  boundary values in $L^p(\mathbb{T})$. \\
 \indent
 On the other hand, since   the local integrability of $|Df|^p$ might fail for  $K$-quasiconformal mappings when  $p\geq 2K/(K-1)$, see for instance Chapter 13 in \cite{AIM}, an analogue of (\ref{area})  is formulated in \cite{AK} as
 \begin{equation}\label{area2}
  \int_{\mathbb{D}} a_f(z)^p (1-|z|)^{p-1}~\ud m<\infty,
 \end{equation}
 where  $a_f (z) = \exp\big(\frac{1}{|B_z|} \int_{B_z} \log Jf^{\frac12}(w)~\ud m \big)$ is  the average derivative function, introduced by Astala and Gehring  in~\cite{AG}. Besides, condition~\eqref{area2} characterizes $\Hp$-quasiconformal mappings with no restriction on $p$ i.e. for all $0<p<\infty$ (see Theorem 5.1 in \cite{AK}).\\
  Note that when $f$ is conformal, $a_f(z)=|f'(z)|$ due to the harmonicity of $\log{|f'|}$, but that is not the case for general analytic functions. If $f$ is quasiregular, then its lack of injectivity leads to difficulties when studying $a_f$.

Analogous to the classical setting, we say that a quasiregular mapping $f\in \mathcal{H}^p_{qr}$ for a given $0<p<\infty$, if
$$
\sup_{0<r<1} ~\frac{1}{2\pi} \int_0^{2\pi}|f (re^{i\theta})|^p ~d\theta= \|f\|_{\mathcal{H}^p}^p < \infty.
$$  
 This class is non-empty, since any bounded quasiregular mapping belongs to $\mathcal{H}^p_{qr}$ for all $0<p<\infty$, because the supremum of the integral means is trivially finite.

 The first part of this note is devoted to provide non-trivial examples of functions in $ \mathcal{H}^p_{qr}$. By the Stoilow factorization theorem, any  quasiregular mapping $f:\D\rightarrow \mathbb{R}^2$ can be written as $f=g\circ \phi$, where $g$ is analytic in $\mathbb{D}$ and $\phi$ is a quasiconformal mapping from $\D$ onto $\D$.
Observe that, since $\phi$ is a quasiconformal selfmapping of $\D$, then $\phi$ has a continuous (even homeomorphic) extension to the boundary $\mathbb{T}$, see the discussion in Chapter 17 in~\cite{va}, in particular~\cite[Theorem 17.8]{va}. In what follows we will denote this extension still by $\phi$. Moreover, $\phi|_{\mathbb{T}}$ is quasisymmetric.

Jerison and Weitsman constructed in~\cite{JW} an example of an analytic function $g\in \mathcal{H}^2$ and a quasiconformal mapping $\phi$ from the disc onto itself such that the quasiregular map $f=g\circ \phi \notin \mathcal{H}^p_{qr}$ for any $p>0$.

In this setting we consider the composition operator  $\mathcal{C}_{\phi}~ g=g\circ \phi$ for $g\in \mathcal{H}^p$, and 
 show that under certain conditions on $\phi$, operator $C_\phi$ sends $\mathcal{H}^p$ to $\mathcal{H}^p_{qr}$.

\begin{theo}\label{thm1} Let  $\phi:\mathbb{D}\rightarrow \mathbb{D}$ be a  quasiconformal mapping and $0<p<\infty$.  Then  $C_\phi: \mathcal{H}^p \rightarrow \mathcal{H}^p_{qr}$ is a bounded operator  if and only if  $\phi^{-1}|_{\mathbb{T}}$ is a Lipschitz function.
\end{theo}

The theorem leads to the following class of mappings: For  $0<p<\infty$ we define the family of quasiregular mappings  $\mathcal{F}_p$ as follows
$$
\mathcal{F}_p:=\{f: \mathbb{D}\rightarrow \mathbb{R}^2; ~ f= g\circ \phi, ~ g\in \mathcal{H}^p ~\textit{and} ~ \phi^{-1}|_{\mathbb{T}} \textit{ is a Lipschitz function} \}.
$$               

 Theorem~\ref{thm1}  shows that $\mathcal{F}_p \subset \mathcal{H}^p_{qr}$. On the other hand, it turns out that not all mappings in $\mathcal{H}^p_{qr}$  can be constructed this way. Indeed, an example is provided by any bounded analytic function precomposed with a quasiconformal mapping whose inverse is not Lipschitz. However, a slightly more subtle example can be given. 
 
 \begin{theo}\label{thm2}

There exists a function $f=g\circ \phi \in \mathcal{H}^1_{qr} $ such that  $\phi^{-1}|_{\mathbb{T}} \textit{ is  Lipschitz} $ but $g\notin \mathcal{H}^1$.
\end{theo}

By arguing as in the proof of Theorem~\ref{thm1}, we recover in Theorem~\ref{thmA} (see section on composition operators) a result for a composition operator from Bergman spaces $\mathcal{A}^p$ to the space of all $p$-integrable quasiregular mappings on $\D$, cf. \cite[Theorem 2]{FGW}. 

In the second part of this note, we study the properties of $\mathcal{H}^p_{qr}$. One cannot expect for a general function $f\in\mathcal{H}^p_{qr}$ to have boundary values a.e on $\mathbb{T}$. In fact, although  bounded planar quasiregular mappings have radial limits in a set $E $ of positive Hausdorff dimension,  this set $ E$ can be of arbitrarily small Hausdorff dimension (see pg. 119-120 in~\cite{N}).\\
\indent
On the other hand,  for quasiregular functions in $\mathcal{F}_p $, we show results analogous to the ones mentioned above in the context of the analytic Hardy spaces and the quasiconformal Hardy spaces.

\begin{theo}\label{thm3}
	Let $0<p<\infty$ and $f$ be a $K$-quasiregular mapping in $\mathcal{F}_p$. Then it holds:
	\begin{itemize}
		\item[(i)]
	$f\in \mathcal{H}^p_{qr}$.
		\item [(ii)]
		The boundary values  $f(\xi)$ exist for a.e. $\xi\in \mathbb{T}$  and   $f(\xi)\in L^p(\mathbb{T})$.
			
			\item [(iii)]
			The non-tangential maximal function  $f^*\in L^p (\mathbb{T})$.
			
				\item [(iv)]
				If  $2\leq p<\frac{2K}{K-1}$, then  it holds that 
				$\int_{\mathbb{D}} |Df(z)|^p~(1-|z|)^{p-1}~\ud m<\infty$.
\end{itemize}

\end{theo}

We remark that part (iv) of the theorem may fail if $0<p<2$, see~\cite[Theorem 1]{Gi}.

{\bf Acknowledgements.} Part of the work was conducted during the Simons semester in \emph{Geometry and analysis in function and mapping theory on Euclidean and metric measure spaces} at IMPAN in fall 2019. The authors would like also to thank Pekka Koskela for many valuable conversations. 

\subsection*{Preliminaries}

In this note, the letter $c$ denotes a positive constant that may change at different occurrences. The notation $A\lesssim  B$ ( $A\gtrsim B$)means that there is a constant $c>0$  such that $A \leq c B$ ($A\geq c B$). The notation $A \simeq B$ means that $A \lesssim B \lesssim A$.

By $\ud m$ we denote the appropriate Lebesgue measure: $1$-, $2$-dimensional, depending on the context of the presentation. 

Moreover, if $g$ is a complex function, then its complex derivative is denoted $g'=g_z$. 
The Jacobian of a map $f$ at a point $w$ is denoted by $Jf(w)$. Note that, if $f$ is holomorphic, then
$Jf(w)=|f'(w)|^2=|f_z(w)|^2$. If $f$ is quasiregular in the plane, then by the distortion inequality it holds that $|Df|\simeq_{c(K)}|f_z|$.

In what follows we often study a quasiconformal selfmapping of $\D$ or an analytic function in $\Hp$ which both have appropriate boundary extensions. We follow a convention to denote these extensions by the same symbols as the given quasiconformal map and the analytic function, respectively.

Finally, let us introduce the notation needed in the proof of Theorem 1.	 For any  $z\in \mathbb{D}$, we denote by $B_z$ the hyperbolic ball $B_z=B(z,c~(1-|z|))$, for some $0<c<1$, and by $I_z$ the interval in $\mathbb{T}$ defined as 
		 \[
		 I_z=\{ e^{i\theta}\in \T: |\arg{ (ze^{-i\theta})}|     \leq \frac{1-|z|}{2}\}.
		 \]
		 Note that $|I_z|\simeq 1-|z|, \textrm{for any}~ z\in \mathbb{D}$.
		 
		Moreover, if $I\subset \mathbb{T}$ is an interval, the Carleson square $S(I)\subset \mathbb{D}$ is the set
		$$
		S(I) = \{re^{i\theta} : e^{i\theta}\in I, 1-\frac{|I|}{2\pi}\leq r<1\},
		$$
		where $|I|$ denotes the length  of the interval $I$.

The following lemma is used in proofs of Theorem~\ref{thm1} and part (iii) of Theorem~\ref{thm3} and stays that an image of a cone under a quasiconformal selfmapping of a unit disc is contained in a cone, whose aperture can be uniformly chosen for all boundary points. This observation is needed in order to apply the non-tangential maximal function on images of cones.
\begin{lem}\label{lem-cones}
 Let $\phi: \D\to\D$ be a $K$-quasiconformal mapping and let $\Gamma(\xi)$ be a cone with a vertex at $\xi\in\T$. Then $\phi(\Gamma(\xi))\subset \Gamma_{\phi(\xi)}$ with the aperture depending only on $K$, the aperture of $\Gamma(\xi)$ and the geometry of $\D$.
\end{lem}	
In the statement above we slightly abuse the term aperture, as cones in our manuscript do not necessarily have the shape of geometric cones.
\begin{proof}

To simplify the notation and the argument, let us replace the unit  disc by the upper half plane. Consider then a cone in the upper half plane $\mathbb{R}^2_+$, with the vertex at the origin:
 $$
 \Gamma=\{z=x+iy\in \mathbb{R}^2_+:~|z|< c~ y\}
 $$
 for a fixed $c>1$, and let $\phi$ be a $K$-quasiconformal map from $\mathbb{R}^2_+$ onto itself with $\phi(0)=0$. 

Let us show that the image of the line $L_+=\{z\in \mathbb{R}^2_+:~|z|= c~ y,~ x>0\}$ is contained in a cone with the vertex at $0$. For that, note that the curve $\gamma=L_+\cup \mathbb{R}_+$ is a quasicircle, and so is its image $\phi(\gamma)= \phi(L_+)\cup \mathbb{R}_+ $.

Let $z\in L_+$ and  $w=\phi(z)$. Denote by $x_0$ the closest point on the real line to $w$, that is if $w=u+iv$, then $v= |w-x_0|$.
 Since the origin lies on the arc of $\phi(\gamma)$ joining $w$ and $x_0$, it holds by the 3-point condition characterizing quasicircles that 
 $$
 |w|\leq |w|+ |x_0| \leq C ~|w-x_0|
 $$
that is $|w| \leq C~v$, where $ C=C(K)$. 

A similar argument shows that the image of the line $L_{-}=\{z\in \mathbb{R}^2_+:~|z|= c~ y,~ x<0\}$ is  contained in the same cone.
\end{proof}
	
\section*{Composition operators with quasiconformal symbol}
		
		The main goal of this section is to prove Theorem~\ref{thm1}, also to show the boundedness of composition operator $C_{\phi}$ on Bergman spaces,  see Theorem~\ref{thmA} below.

		\begin{proof}[Proof of Theorem 1]
			
			Assume that $C_\phi: \mathcal{H}^p \rightarrow \mathcal{H}^p_{qr}$ is a bounded operator, that is for any $g\in  \mathcal{H}^p$, 
			\begin{equation}\label{bound1}
		\sup_{r<1} \int_{\mathbb{T}}|g(\phi(rz))|^p~\ud m \leq c \int_{\mathbb{T}}|g(z)|^p~\ud m.
			\end{equation}

			Set $g(z)= \frac{1}{(1-\bar{w}z)^{2/p}}$, with $w\in\mathbb{D}$. Then, $g\in \mathcal{H}^p$, and  $\|g\|^p_{\mathcal{H}^p}\simeq\frac{1}{1-|w|^2}$.

			Next, by Fatou's lemma and~\eqref{bound1}
				\begin{equation}\label{bound}
				\begin{split}	
			\int_{\mathbb{T}} \frac{1}{|1-\bar{w}\phi(z)|^2}~\ud m(z)&\leq \liminf_{r\rightarrow1} \int_{\mathbb{T}} \frac{1}{|1-\bar{w}\phi(rz)|^2}~\ud m(z) \\
			&\leq \sup_{r<1}\int_{\mathbb{T}} \frac{1}{|1-\bar{w}\phi(rz)|^2}~\ud m(z)
			 \leq c~ \frac{1}{1-|w|^2}.
			 \end{split}
			\end{equation}

			Define the measure $\mu$ on ${\mathbb{T}}$ as the pushforward of the Lebesgue measure on $\mathbb{T}$ via $\phi$, that is let $\mu (E):=m (\phi^{-1}(E))$, for any set $ E\subset \mathbb{T}$. Then, by the change of variables formula, \eqref{bound} can be expressed in terms of $\mu$  as follows
			\begin{equation}\label{mu}
			\int_{\mathbb{T}} \frac{1}{|1-\bar{w}z|^2}~\ud\mu(z) \leq c~ \frac{1}{1-|w|^2}.
			\end{equation}
		
			If $z\in I_w$, then it holds that  
			\begin{equation}\label{thm1-ineq-aux}
			|1-\bar{w}z| \lesssim 1-|w| \simeq |I_w|.
			\end{equation}
			
			 Therefore, by \eqref{mu} and~\eqref{thm1-ineq-aux}  we conclude that $\mu(I_w)\leq c ~|I_w|$. Hence, by the definition of $\mu$
			$$
			|\phi^{-1}(I)| \leq c~|I|\quad \textrm{for any interval}~I\subset \mathbb{T},
			$$
			as any $I=I_w$ for some $w\in\mathbb{D}$. This shows that function $\phi^{-1}$ is Lipschitz.
			
			In order to prove the converse assertion, recall Lemma~\ref{lem-cones} and note that  if $\phi^{-1}$ is Lipschitz, then by a change of variables
			\begin{equation*}
			\begin{split}
				\sup_{r<1} &\int_{\mathbb{T}}|g(\phi(rz))|^p~\ud m(z) \leq  \int_{\mathbb{T}}|(g\circ \phi)^*(z)|^p~\ud m(z) 
				\lesssim  \int_{\mathbb{T}}|g^*(\phi(z))|^p~\ud m(z)\\ &
				 	 	= \int_{\mathbb{T}}|g^* (w)|^p~|(\phi^{-1})' (w)|~\ud m(w)
				 	 	\leq c
				 	 	\int_{\mathbb{T}}|g^* (w)|^p ~~\ud m(w)\leq c ~\|g\|_{\mathcal{H}^p}^p.
				 	 	\end{split}
				\end{equation*}
		\end{proof}
		
		Using similar arguments, we can  recover the following result (see~\cite[Theorem 2]{FGW} for $p=2$ and Remark 1 on pg. 10 in~\cite{FGW} for all $0<p<\infty$). Recall that the Bergman space, denoted $\mathcal{A}^p$, consists of all holomorphic functions $p$-integrable with respect to the Lebesgue measure on $\D$. Furthermore, by the analogy to space $\Hp_{qr}$ we denote by  $\mathcal{A}^p_{qr}$ the space of all $p$-integrable quasiregular mappings on $\D$.
				
		\begin{theorem}[FGW]\label{thmA}
			The operator $ C_\phi$  is a bounded operator from $\mathcal{A}^p$ to $\mathcal{A}^p_{qr }$ for $0<p<\infty$, if and only if  $\phi^{-1}|_\mathbb{T}$ is Lipschitz.
		
		\end{theorem}
	
	\begin{proof}
		Recall that Carleson measures for Bergman spaces  are given by the condition 
		\begin{equation}\label{thmA-Carleson}
		\mu(B) \leq c ~|B|, \hbox{ for any hyperbolic ball }B\subset \D.
		\end{equation}
		Namely, Theorem 2.2 in~\cite{L} applied with $n=0$, $\alpha=0$ and $p=q>0$ stays that $\|g\|_{L^p(\ud\mu)}\leq C \|g\|_{\mathcal{A}^p}$ if and only if $\mu(B(z, r)) <Cr^2$, for all $z \in \D$ and $r=\frac12 \dist(z, \partial \D)$. This can be shown to be equivalent to a condition $\mu(S(I))<c |I|^2$, for any interval $I\subset \T$ and the corresponding Carleson square $S(I)$. 
		
		The condition $C_\phi$ bounded means
		\begin{equation}\label{thmA-cond}
		\int_{\mathbb{D}}|g(\phi(z))|^p~\ud m \leq c~\|g\|_{\mathcal{A}^p}^p.
		\end{equation}
		Let $\mu$  denote the pushforward measure of the Lebesgue measure on $\mathbb{D}$ under mapping $\phi$. Then, by the change of variables formula, the boundedness condition~\eqref{thmA-cond} can be equivalently written as 
		$$
		\int_{\mathbb{D}}|g(z)|^p~\ud\mu \leq c~\|g\|_{\mathcal{A}^p}^p
		$$
		which in turn, by~\eqref{thmA-Carleson},  holds if and only if 
		\begin{equation}\label{thmA-Carlseon2}
		|\phi^{-1}(B)| \leq c ~|B| \hbox{ for any hyperbolic ball }B\subset \D.
		\end{equation}
		 As we will show next, this  is equivalent to saying that $\phi^{-1}|_\mathbb{T}$ is Lipschitz.
		
Denote by $w$ the center of the ball $B$, and let $\phi(z)=w$. Recall the following statement of the so-called circular distortion theorem for $K$-quasiconformal mappings specialized to our case of $\phi:\D\to\D$ (see e.g.~\cite[Lemma 2.1]{AK}): there exists a constant $C=C(K)$ so that for all $z\in \D$
			\begin{equation}\label{circ-dist}
			 \diam \phi^{-1}(B)\simeq 1-|z|\quad\hbox { and }\quad \phi^{-1}(B)\supset B(z, 1/C (1-|z|)).
			\end{equation}
		Therefore, $ |\phi^{-1}(B)|\simeq (1-|z|)^2$, and so condition~\eqref{thmA-Carlseon2} above reads $1-|z| \leq c~(1-|w|)$.

		Let $J$ be any interval on $\T$,  $w_0\in \mathbb{D}$ be the point such that $I_{w_0}=J$, and  $z_0= \phi^{-1}(w_0)$. Then, by the circular distortion theorem,  the interval
		$I = \phi^{-1}(J)$ contains and is contained in bounded multiples of $I_{z_0}$. Since $|I_{w_0}|\simeq 1-|w_0|$ and similarly for $|I_{z_0}|$, we can conlude that 
		$$
		|I| \simeq 1-|z_0|\leq c~(1-|w_0|)  \simeq |J|
		$$
		that is,  $\phi^{-1}|_\mathbb{T}$ is Lipschitz. 
		
		In order to show the opposite implication, notice that by assuming the Lipschitz regularity of $\phi^{-1}|_\mathbb{T}$, we may reverse the presented reasoning and obtain~\eqref{thmA-Carlseon2}.
			\end{proof}
		
	\section*{Quasiregular mappings $\mathcal{H}^p_{qr}$}
		 
		The first question we would like to address is whether any function in  $\mathcal{H}^p_{qr}$ belongs to the family $\mathcal{F}_p$, see also the discussion before the statement of Theorem~\ref{thm2} in Introduction.
		
		\begin{proof}[Proof of Theorem~\ref{thm2}] Let us consider the following homeomorphism  $\phi:[-\pi, \pi] \rightarrow [-\pi, \pi]$, defined as
		\begin{equation*}
			\alpha (t) = \left\{
			\begin{array}{ll}
				\sqrt{\pi t}     &~~ 0\leq t \leq \pi \\
				-\sqrt{-\pi t} & -\pi \leq t \leq 0 \\
					\end{array}
			\right.
		\end{equation*}
		
		It is easy to check that the function $\phi(t)=e^{i\alpha(t)}$ is  quasisymmetric  on $\mathbb{T}$. 
	Therefore, $\phi$ can be extended to a quasisymmetric mapping on the closure of $\mathbb{D}$, which will be denoted by $\phi$ as well.
		
		Next, set $g(z)=\frac{1}{1-z}$. Then $g$ is analytic on ~$\mathbb{D}$  and, since $\int_{0}^{2\pi} \ud \theta/|1-e^{i\theta}|$ diverges,   $g \notin \mathcal{H}^1$.  On the other hand, $f=g\circ \phi \in \mathcal{H}^1_{qr}$ by the following reasoning.
		
	First, recall that by Lemma~\ref{lem-cones} cones are mapped to cones and hence $|g(\phi(r\xi))|\leq g^*(\phi(\xi))$. Let us define $g(\xi)=\frac{1}{1-\xi}$ for $\xi\in \T$ with convention that $g(1):=\infty$. Then it holds that $g^*(\xi) \simeq |g(\xi)|$ for any $\xi\in \T$, because if $z\in \Gamma_\xi$ then $ |1-\xi| \lesssim |1-z|$. Therefore,
	$$
\int_{\mathbb{T}} |f^*|~\ud m \lesssim \int_{\mathbb{T}} |g^*\circ \phi|~\ud m=2
\int_0^\pi|g^* (e^{i \alpha(t)})|~\ud t.
$$
We apply the change of variables $s=\alpha(t)$ and get $s~\ud s=\frac{\pi}{2}~\ud t$ for $0<t<\pi$. Thus,
\begin{align*}
\int_0^\pi|g^* (e^{i \alpha(t)})|~\ud t & \simeq \int_0^\pi |g^*(e^{is})| s~\ud s\simeq
\int_0^\pi \frac{s}{|1-e^{is}|}~\ud s \\
&\lesssim \int_0^{\pi/2}\frac{s}{\sin{s}}~\ud s+  \int_{\pi/2}^{\pi}\frac{1}{|1-e^{is}|}~\ud s<\infty.
\end{align*}
%
%
\end{proof}
		
   Finally, we present the proof of Theorem~\ref{thm3}.
	\begin{proof}[Proof of Theorem 3]
		\begin{enumerate}
			\item [(i)]  It is a consequence of Theorem 1.
		\item [(ii)]
		 Let $f=g\circ\phi$. Since $g\in \mathcal{H}^p$, the boundary values $g(e^{i\theta})$ exist for every $\theta\in \mathbb{T}\backslash E$ where $E$ is a set of measure $0$. Recall further, that $\phi$ has a homeomorphic extension to $\overline{\D}$. Therefore, $f$ has boundary values on  $\mathbb{T}\backslash \phi^{-1} (E)$. Since $\phi^{-1}|_\mathbb{T}$ is Lipschitz regular, we conclude that  $| \phi^{-1} (E)|=0$ and, hence, $f$ has boundary values a.e. on $\mathbb{T}$.\\
		\noindent
Furthermore, a change of variables allows us to conclude the integrability of $f$:
		\begin{equation*}	
		\int_{\mathbb{T}} |f|^p~\ud m = \int_{\mathbb{T}} |g\circ \phi|^p~\ud m = \int_{\mathbb{T}} |g|^p ~|(\phi^{-1})'|~dm\leq c ~ \|g\|_{\mathcal{H}_p}^p. 
		\end{equation*}
		
		\item[(iii)] Since, by Lemma~\ref{lem-cones},  the image of a cone $\Gamma(z)$ at a point $z\in \mathbb{T}$ is contained in a cone at $\phi(z)$, we have
		$$
		\int_{\mathbb{T}} |f^*|^p~dm \lesssim \int_{\mathbb{T}} |g^*\circ \phi|^p~dm=
		\int_{\mathbb{T}} |g^*|^p ~|(\phi^{-1})'|~dm\leq c ~ \|g\|_{\mathcal{H}_p}^p. 
		$$
	
		\item [(iv)] 
		We start by recalling Theorem 3.1 in \cite{L}, here stated for $p=q\geq 2$ and $n=1$.
		
		\emph{
		If $g\in \mathcal{H}^p$ and $\mu$ is a positive measure in $\mathbb{D}$, then
		$$
		\int_{\mathbb{D}} |g'|^p ~d\mu \leq c \|g\|_{\Hp}^p 
		$$
		if and only if $\mu(S)\leq c ~(l(S))^{1+p}$, for any Carleson square $S$.}\\
		It is easy to show that the condition on the measure can be replaced by  
		\[
		\mu(B)\leq c~r_B^{1+p},
		\]
		 for any hyperbolic ball $B$ of radius $r_B$. This occurs when the exponent is strictly bigger than $1$.
		
			Upon defining a measure $\mu$ as the pushforward of the measure $|\phi'(z)| ^p~(1-|z|)^{p-1} ~\ud m(z)$, we may write
		\begin{align*}
		\int_{\mathbb{D}} |Df(z)|^p ~(1-|z|)^{p-1}~\ud m(z)&\lesssim \int_{\mathbb{D}} |g'(\phi(z))|^p ~|\phi'(z)| ^p  ~(1-|z|)^{p-1}  ~\ud m(z)\\
	&	=\int_{\mathbb{D}} |g'(z)|^p~\ud\mu(z).
		\end{align*}
		Here we also use that for a $K$-quasiconformal mapping $\phi$, it holds that $|D\phi|\leq c(K) |\phi'|$. Therefore, the assertion (iv) will be proven if we show that for any hyperbolic ball B of radius $r_B$ 
		$$
		\mu(B)=\int_{\phi^{-1}(B)} |\phi'(z)|^p  ~(1-|z|)^{p-1}~\ud m(z)\leq c~r_B^{1+p}.
		$$
		
		By the circular distortion theorem, we have that set $\mathbb\phi^{-1}(B)\subset \tilde{B}$ for some hyperbolic ball $\tilde{B}$, and  $\phi(\tilde{B}) $ is contained in another hyperbolic ball centered at the same point as $B$ and radius comparable to $r_B$ (see~\cite[Lemma 2.1]{AK}, also the discussion at~\eqref{circ-dist} above applied for $\phi$ and $\phi^{-1}$). In particular, $|\phi(\tilde{B})|\leq |c B|\simeq r_B^2$.
		
		 Next, observe that for $z\in \tilde{B}=\tilde{B}(z_0,r_{\tilde{B}})$ it holds that
		 $r_{\tilde{B}}\simeq1-|z|$. This combined with estimates on the integrability of the derivatives of a $K$-quasiconformal map (see \cite[Corollary 13.2.4]{AIM}) when $p< 2K/(K-1)$ gives us that
		\begin{align*}
		\mu(B) &\leq \int_{\tilde{B}} |\phi'(z)|^p  ~(1-|z|)^{p-1}~\ud m(z) \\
		 & \simeq r_{\tilde{B}}^{p-1}\int_{\tilde{B}}|\phi'(z)|^p ~\ud m(z) \\
		 & \leq c(K, p)~r_{\tilde{B}}^{p-1}~\left(\frac{|\phi(\tilde{B})|}{|\tilde{B}|}\right)^{\frac{p}{2}}|\tilde {B}| 
		 \lesssim r_{\tilde{B}}^{p-1}\left(\frac{r_B^2}{r_{\tilde{B}}^2}\right)^{\frac{p}{2}} r_{\tilde{B}}^2
		\lesssim  r_{\tilde{B}}~r_B^p.
		\end{align*}
		Finally, by reasoning  as in the proof of Theorem A, the fact that $\phi^{-1}|_\mathbb{T}$ is Lipschitz implies that 	$ r_{\tilde{B}}\lesssim r_B$, see the discussion at~\eqref{thmA-Carlseon2} and~\eqref{circ-dist}. Hence, we conclude that $\mu (B)\leq r_B^{p+1}$, as we wanted to show. This completes the proof of assertion (iv) of the theorem.	
	
	\end{enumerate}	
	\end{proof}

\textit{Tomasz Adamowicz:} Institute of Mathematics, Polish Academy of Sciences, ul. \'Sniadeckich 8, Warsaw 00-656, Poland. E-mail address: tadamowi@impan.pl

\textit{Mar\'ia J. Gonz\'alez:} Departamento de Matem\'aticas, Universidad de C\'adiz, 11510 Puerto Real (C\'adiz), Spain. E-mail address: majose.gonzalez@uca.es

\end{document}